\documentclass[11pt,a4paper]{amsart}
\usepackage{amsmath}
\usepackage[utf8]{inputenc}
\usepackage[english]{babel}
\usepackage[T1]{fontenc}
\usepackage{url}
\usepackage{mathrsfs}
\usepackage{ifthen}
\usepackage{tikz}
\usepackage{ amssymb }
\usepackage{lscape}
\usepackage{multirow}

\usepackage{caption}

\title[The 2-colored operad on the set of singular meanders]{The 2-colored operad on the set of singular meanders}
\author{Yury Belousov}
    \thanks{This research was supported by the Ministry of Science and Higher Education of the
Russian Federation, agreement 075-15-2019-1620 date 08/11/2019 and 075-15-2022-289 date 06/04/2022.
The author is partially supported by International Laboratory of Cluster Geometry NRU HSE, RF Government grant, ag. № 075-15-2021-608 from 08.06.2021}

\address{Yury Belousov\\
Faculty of Mathematics, National Research University HSE}
\email{bus99@yandex.ru}

\newcommand \ifdraw {\iftrue}

\newcommand \R  {\mathbb R}
\newcommand \M  {\mathcal{M}}
\newcommand \ord{\operatorname{ord}}

\newtheorem{thm}{Theorem}

\newtheorem{lemma}{Lemma}

\newtheorem{corollary}{Corollary}

\theoremstyle{definition}
\newtheorem{definition}{Definition}

\theoremstyle{remark}
\newtheorem*{remark}{Remark}

\begin{document}
\maketitle

\begin{abstract}
In the recent work~\cite{B21} we introduced a geometric decomposition of meanders. In the present paper, we generalize this approach to the case of singular meanders and give a more algebraic description of this decomposition.
\end{abstract} 

\maketitle

\section{Introduction}
This work is devoted to the study of singular meanders (examples of singular meanders can be found in Figures~\ref{fig: examples of meanders} and~\ref{fig: example of irreducible}). 
In the recent work~\cite{B21} we introduced new geometric decomposition of meanders. The key ingredients in this decomposition are the operations of insertion of one meander into another (see Definition~\ref{def:insert}). In the present paper we use these operations to define the structure of 2-colored operad on the set of singular meanders (Theorem~\ref{thm:operad}), and we show that each singular meander admits a canonical decomposition in terms of these operations (Theorem~\ref{thm:decompose}). As a corollary we find an equation on the generating function for singular meanders (Corollary~\ref{cor:gen function equation}). 

\subsection*{Acknowledgment}
The author thanks Andrei Malyutin for valuable suggestions and comments.

%\section{Introduction}

%In the recent work~\cite{B21} we introduced geometric decomposition of meanders. In the following paper we are going to generalize this approach for the case of singular meanders and give more algebraic description of this decomposition.

%We will use the definitions from the~\cite{B21}[Sec.~2] 

%We will mostly use the same terminology as in~~\cite{B21}. The only difference is that we are now using a notion of singular meanders (see examples on fig.~\ref{fig:singular meanders}). 

\section{Basic definitions}\label{sec: basic definitions}
The definitions below are reformulations of the definitions from the work~\cite[sec.~1]{B21} to the case of singular meanders. 
\begin{definition}
A \emph{singular meander} $(D, \{p_1, p_2,p_3,p_4\}, \{m, l\})$ is a triple of \begin{itemize}
    \item euclidean 2-dimensional disk $D$;
    \item four distinct points $p_1, p_2,p_3,p_4$ on the boundary $\partial D$ such that there exists a connected component of $\partial D \setminus \{p_1, p_2\}$ containing $\{p_3,p_4\}$;
    \item the images $m$ and $l$ of smooth proper embeddings of the segment $[0;\,1]$ into $D$ such that $\partial m = \{p_1,p_3\}$, $\partial l = \{p_2,p_4\}$, and $m$ and $l$ intersect (not necessary transversely) in a finite number of points.
\end{itemize}
If all the intersections between $m$ and $l$ are transverse we say that $M$ is a \emph{non-singular meander}. 
\end{definition}

\begin{definition}
We say that two singular meanders $$M=(D,\{p_1,p_2,p_3,p_4\}, \{m,l\})$$ and $$M'= (D',\{p_1',p_2',p_3',p_4'\},\{m',l'\})$$ are \emph{equivalent} if there exists a homeomorphism $f:D\to D'$ such that $f(m)=m'$, $f(l)=l'$, and $f(p_i)=p_i'$ for each $i=1,\dots,4$. 
\end{definition}
\begin{remark}
We will always draw singular meanders in such a way that $D$ is a euclidean disk in $\R^2$, $l$ is the horizontal diameter in $D$, $p_2$ is the left endpoint of $l$, and $p_1$ is always drawn above $p_2$. That's why we do not place $l$, $m$, $p_1,p_2,p_3,p_4$ in figures. Examples of singular meanders are given in Figures~\ref{fig: examples of meanders} and~\ref{fig: example of irreducible}.
\end{remark}

\begin{figure}[h]
    \centering
    \begin{tikzpicture}[scale = 1.3]
        \draw[thick] (0, 0) to (3, 0);
        \draw[ultra thick] (0.0476312, 0.375) to[out = 0, in = 90, distance = 32.9867] (2.625, 0)
         to[out = -90, in = -70, distance = 5.71239] (2.25, -0.12)
         to[out = 120, in = 70, distance = 4.71239] (1.875, -0.12)
         to[out = -120, in = -90, distance = 5.71239] (1.5, 0)
         to[out = 90, in = 90, distance = 4.71239] (1.125, 0)
         to[out = -90, in = -90, distance = 4.71239] (0.75, 0)
         to[out = 90, in = 90, distance = 4.71239] (0.375, 0)
        to[out = -90, in = 180, distance = 32.9867] (2.95237, -0.375);
        \draw[help lines] (1.5, 0) circle (1.5);
        \draw[fill] (0.0476312, 0.375) circle (0.05);
        \draw[fill] (2.95237, -0.375) circle (0.05);
        \draw[fill] (0, 0) circle (0.05);
        \draw[fill] (3, 0) circle (0.05);
    \end{tikzpicture}
    \hspace{1.5cm}
    \begin{tikzpicture}[scale = 1.3]
        \draw[thick] (0, 0) to (3, 0);
        \draw[ultra thick] (0.163915, 0.681818) to[out = 0, in = 90, distance = 5.14079] (0.409091, 0)
         to[out = -90, in = -90, distance = 1.7136] (0.545455, 0)
         to[out = 90, in = 90, distance = 5.14079] (0.954545, 0)
         to[out = -90, in = -90, distance = 11.9952] (1.90909, 0)
         to[out = 90, in = 90, distance = 1.7136] (2.04545, 0)
         to[out = -90, in = -90, distance = 15.4224] (0.818182, 0)
         to[out = 90, in = 90, distance = 1.7136] (0.681818, 0)
         to[out = -90, in = -90, distance = 18.8496] (2.18182, 0)
         to[out = 90, in = 90, distance = 5.14079] (1.77273, 0)
         to[out = -90, in = -90, distance = 1.7136] (1.63636, 0)
         to[out = 90, in = 90, distance = 8.56798] (2.31818, 0)
         to[out = -90, in = -90, distance = 25.7039] (0.272727, 0)
         to[out = 90, in = 90, distance = 1.7136] (0.136364, 0)
         to[out = -90, in = -90, distance = 29.1311] (2.45455, 0)
         to[out = 90, in = 90, distance = 11.9952] (1.5, 0)
         to[out = -90, in = -90, distance = 5.14079] (1.09091, 0)
         to[out = 90, in = 90, distance = 22.2767] (2.86364, 0)
         to[out = -90, in = -90, distance = 1.7136] (2.72727, 0)
         to[out = 90, in = 90, distance = 18.8496] (1.22727, 0)
         to[out = -90, in = -90, distance = 1.7136] (1.36364, 0)
         to[out = 90, in = 90, distance = 15.4224] (2.59091, 0)
        to[out = -90, in = 180, distance = 5.14079] (2.72643, -0.863636, 0);
        \draw[help lines] (1.5, 0) circle (1.5);
        \draw[fill] (0.163915, 0.681818) circle (0.05);
        \draw[fill] (2.72643, -0.863636) circle (0.05);
        \draw[fill] (0, 0) circle (0.05);
        \draw[fill] (3, 0) circle (0.05);
    \end{tikzpicture}
    \caption{Examples of singular meanders}
    \captionsetup{aboveskip=0pt,font=it}
    \label{fig: examples of meanders}
\end{figure}
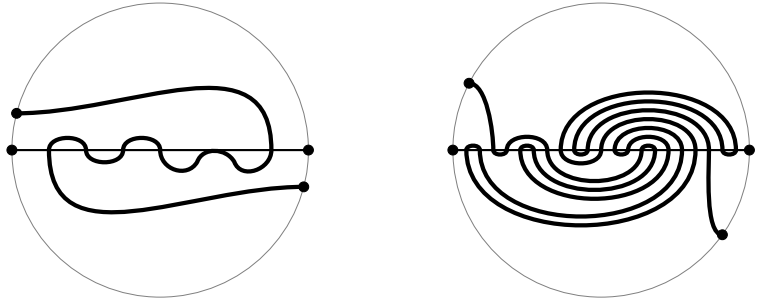

\begin{definition}
Let $M=(D, \{p_1, p_2,p_3,p_4\}, \{m, l\})$ be a singular meander, let $n_{\mathrm{t}}(M)$ be the number of transverse intersections of $m$ and $l$, and let $n_{\mathrm{nt}}(M)$ be the number of non-transverse intersections of $m$ and $l$. The \emph{order of $M$} (denoted by $\ord (M)$) is the pair $(n, k) := \left(\max_{M'\in [M]} n_{\mathrm{t}}(M'), \min_{M'\in [M]} n_{\mathrm{nt}}(M')\right)$, where $[M]$ is the set of all singular meanders that are equivalent to $M$. If the order of $M$ is $(n, k)$ then the \emph{total order} of $M$ is $n+k$. By $\M_{n, k}$ we denote the set of all equivalence classes of singular meanders of order $(n, k)$. 
%Let $[M]$ be a set of all equivalence classes of a singular meander $M=(D, \{p_1, p_2,p_3,p_4\}, \{m, l\})$. Then by the \emph{order of $M$} (denoted by $\ord(M)$) we mean the pair $(n, k)$, where $n$ is maximum number of transverse intersection points of $m$ and $l$ among all the elements in $[M]$ and $k$ is the number of intersection points of $m$ and $l$ minus $n$.By $\M_{(n, k)}$ we denote the number of all equivalence classes of meanders of order $(n, k)$. 
\end{definition}

\begin{remark}
Without loss of generality further in this paper we always assume that if we are given a singular meander $M$ of order $(n,k)$, we have $n_{\mathrm{t}}(M)=n$ and $n_{\mathrm{nt}}(M) = k$.
\end{remark}

\begin{definition}
We say that a singular meander $M'=(D', \{p_1', p_2',p_3',p_4'\}, \{m', l'\})$ is a \emph{submeander} of a singular meander $M=(D, \{p_1, p_2,p_3,p_4\}, \{m, l\})$ if \begin{itemize}
    \item $D' \subseteq D$;
    \item $m' = D' \cap m$;
    \item $l' = D' \cap l$;
    \item $p_1' = \gamma\left(t\right)$, where  $\gamma: [0;1] \to D$ is any continuous map such that $\gamma([0;1]) = m$, $\gamma(0) = p_1$, and $t=\min \{x\in [0;1]\ |\ \gamma(x) \in D'\}$;
    \item If $\ord(M') = (2n+1, k)$ then $p_2'= \gamma'\left(t'\right)$, where  $\gamma': [0;1] \to D$ is any continuous map such that $\gamma'([0;1]) = l$, $\gamma'(0) = p_2$, and ${t'=\min \{x\in [0;1]\ |\ \gamma'(x) \in D'\}}$.
\end{itemize}
\end{definition}

\begin{definition}
Let $$M'=(D', \{p_1', p_2',p_3',p_4'\}, \{m', l'\})$$ and  $$M''=(D'', \{p_1'', p_2'',p_3'',p_4''\}, \{m'', l''\})$$ be two submeanders of a singular meander $$M=(D, \{p_1, p_2,p_3,p_4\}, \{m, l\}).$$ We say that $M'$ and $M''$ are \emph{equivalent with respect to $M$} if $D'\cap m\cap l = D'' \cap m \cap l$.
\end{definition}
\begin{remark}
Notice that submeanders $M'$ and $M''$ of $M$ can be equivalent as singular meanders, but non-equivalent as submeanders with respect to $M$. But if $M'$ and $M''$ are equivalent with respect to $M$ they are also equivalent. 
\end{remark}

\begin{remark}
The notion of submeanders can be used to define a category on the set of all singular meanders. We define this category as follows. The set of objects is just the set of all classes of equivalence of singular meanders, and the set of morphisms from object $m$ to $m'$ is the set of submeanders of $m'$ that are equal (as singular meanders) to $m$. In this category all small limits and small colimits exist. Moreover, a singular meander of total complexity zero is the initial object in this category. 
\end{remark}

%The notion of submeanders is useful to define the insertion of me

\begin{definition} \label{def:insert}
Let $$M = (D, \{p_1, p_2,p_3,p_4\}, \{m, l\})$$ and $$M' = (D', \{p_1', p_2',p_3',p_4'\}, \{m', l'\})$$ be two singular meanders of order $(n, k)$ and $(n',k')$ respectively, and let 
$$M'' = (D'', \{p_1'', p_2'',p_3'',p_4''\}, \{m'', l''\})$$
be a submeander of $M$ such that (i) $\ord(M'') = (n'',k'')$, (ii)
$n' \equiv n''\ \mathrm{mod}\ 2$. 
Consider a map $f:\partial D'' \to \partial D'$ such that $f({p}_i'')=p'_i$ for each $i=1,\dots,4$. 
There is a well-defined singular meander $$\tilde{M} = (\tilde{D}, \{p_1, p_2, p_3, p_4\}, \{\tilde{m}, \tilde{l}\})$$ where  
\begin{itemize}
    \item $\tilde{D} = \big(D\setminus \operatorname{Int}(D'')\big) \cup_f D'$;
    \item $\tilde{m} = \big(m\setminus \operatorname{Int}(D''\cap m)\big)\cup_f m'$;
    \item $\tilde{l} = \big(l\setminus \operatorname{Int}(D''\cap l)\big)\cup_f l'$.
\end{itemize}
We say that $\tilde{M}$ is obtained by the \emph{insertion of $M'$ into $M$ at $M''$}. 
%If the total order of $M''$ is $1$ we say that $\tilde{M}$ is obtained by the \emph{primary insertion of $M'$ into $M$ at $M''$}. 
%If the total order of $M'$ is $1$ we say that $\tilde{M}$ is obtained by the \emph{cut of $M''$ from $M$}.
\end{definition}

%For a given singular meander $M$ the set of equivalence classes of it's submeanders forms a poset. The partial order is defined as follows. Let $M', M''$ be two non-equivalent submeanders of $M$, and $m'$ and $m''$ be an equivalence classes of $M'$ and $M''$ respectively.  We write that $m' \preccurlyeq m''$ if there exist singular meanders $\tilde{M}' \in m''$ such that $\tilde{M}'$ is a submeander of $M''$. The Hasse diagram of this poset is called the \emph{decomposition graph} of $M$.

%\begin{definition}
%Let $M$ be a singular meander, and let $G(M)$ be a set of all submeanders of $M$ up to equivalence with respect to $M$. Let $m', m'' \in G(M)$. We write that $m' \preccurlyeq m''$ if there exist singular meanders $M' \in m'$ and $M'' \in m''$ such that $M'$ is a submeander of $M''$. 
%The \emph{decomposition graph} of $M$ is the graph representing the Hasse diagram of the poset $(G(M), \preccurlyeq)$.
%\end{definition}

%\begin{remark}
%We will always draw decomposition graphs as plane graphs on $\R^2$, where the $y$-coordinate of a vertex $v$ is equal to the total order of the corresponding class of singular meanders. 
%Examples of such graphs for two singular meanders are given in Figure~\ref{fig: example of graphs}.
%\end{remark}

\section{Operad}
The insertion of a singular meander can be viewed as an operation on the set of singular meanders. To define this operation, we need to label all of the intersection points of a given singular meander. We will do this as follows. Let  
$$M = (D, \{p_1, p_2,p_3,p_4\}, \{m, l\})$$
be a singular meander of order $(n, k)$. Consider a bijective map $\gamma: [0;\,n+1] \to D$, such that 
(i) $\gamma([0;\, n+1]) = l$, 
(ii) $\gamma(0) = p_2$,
(iii) $\gamma(t)$ is an intersection of $m$ an $l$ if and only if $t \in \{1,2,\dots n\}$. We say that an intersection point $p$ of $m$ and $l$ has the label $k$ if $\gamma(k) =p$.
Notice that this ordering of intersection points does not depend on the choice of $\gamma$. Analogously, we define the labeling on the non-transverse intersections of $m$ and $l$.

Now we define the operations on the set of all singular meanders.
Let $M$ be a singular meander of order $(n, k)$ and let $M|_i$ be a submeander of $M$ with the only transverse intersection point with label $i$. Let $M'$ be a singular meander of order $(2n'+1, k')$. 
Then let $M\circ_i M'$ be a singular meander of order $(n+2n', k+k')$ obtained by the insertion of $M'$ into $M$ at a singular meander $M|_i$. Notice that these operations give well-defined operations on the set of all equivalence classes of singular meanders. 
Analogously, we can define an operation $M \bullet_i M'$ --- the insertion of singular meander $M'$ of order $(2n', k')$  into $M$ at the singular meander $M|_i$, where $M|_i$ is the submeander of $M$ with the only non-transverse intersection point with label $i$.

The straightforward check shows that these operations form a 2-colored operad (for the definition of a colored operad see, for example, ~\cite{LV12, Y16}). 
%We formulate this as a theorem.
\begin{thm}\label{thm:operad}
The set $\mathcal{M} = \bigcup\limits_{n\geq 0, k \geq 0} \mathcal{M}_{n,k}$ together with the set of operations
\begin{align}\label{operation 1}
    &\circ_i: \M_{n,k} \times \M_{2n'+1, k'} \to \M_{n+2n', k+k'} &  n \geq 1,\ k \geq 0,\ 1\leq i\leq n,\\
\label{operation 2} &\bullet_i: \M_{n,k} \times \M_{2n', k'} \to \M_{n+2n', k+k'-1}  &n \geq 0,\ k \geq 1,\   1\leq i \leq k.
\end{align}
form a 2-colored operad.% on the set of singular meanders. 
\end{thm}

\section{Unique decomposition into prime factors}
In the work~\cite{B21} it was shown that any meander admits a canonical decomposition into snakes and irreducible meanders. Now we are going to generalize this result to the case of singular meanders. For this we need to define two building blocks of this decomposition: iterated snakes and irreducible meanders. 

Later in the paper we use the following notation. Let $\{A_{n,k}\}_{n\geq 0,\, k\geq 0}$ be some sequence of numbers, and let $f(x,t) = \sum\limits_{n\geq 0, k\geq 0} A_{n,k}x^n t^k$ be the generating function for this sequence. 
We will decompose $f(x,t)$ in two parts (the <<odd>> part $f^{(1)}(x,t)$ and the <<even>> part $f^{(2)}(x,t)$) in the following way:
$$
f(x, t) = \underbrace{\sum\limits_{n\geq 0, k\geq 0} A_{2n+1,k}x^{2n+1} t^k}_{=:f^{(1)}(x,t)} + \underbrace{\sum\limits_{n\geq 0, k\geq 0} A_{2n,k}x^{2n} t^k}_{=:f^{(2)}(x,t)}.
$$

Another notation that we are going to use is the following. Let $f(x,t)$ and $g(x,t)$ be generating functions of some sequences. Then we will use the following notation 
$$
\left(f \boxdot g\right) (x,t):= f\left(g^{(1)}(x,t),\, g^{(2)}(x,t)\right).
$$ 

\subsection{Iterated snakes}
\begin{definition}
We say that a singular meander $M$ of total order $N$ is a \emph{snake}, if there are precisely $\frac{N(N+1)}{2} + 1$ submeanders of $M$ that are pairwise non-equivalent with respect to $M$.
\end{definition}

Class of snakes can be easily described by the following lemma from the work~\cite{B21}.

\begin{lemma}[\cite{B21}]
Let $M$ be a singular meander of total order $N$. Then $M$ is a snake if and only if its permutation is either $(1, 2, \dots, N)$ or $(N, N-1, \dots , 1)$.
\end{lemma}

Let us denote the number of equivalence classes of snakes of order $(n,k)$ by $M_{n,k}^{(S)}$. Then we have 
$$
M_{n,k}^{(S)} = \begin{cases}
0 & n=0 \text{ and } k=0, \\
1 & n=1 \text{ and } k=0, \\
\binom{n+k}{k} & n \text{ is even,} \\
2\binom{n+k}{k} & \text{otherwise}.
\end{cases}
$$
Now it is easy to write the generating function for the number of snakes.
\begin{corollary}
Let $\phi_{(S)}(x,t)$ be the generating function for the number of snakes. Then
$$
\phi_{(S)}(x,t) = \sum_{n\geq 0, k \geq 0}M_{n,k}^{(S)}x^nt^k =-x+ \frac{-t^2+t+x (x+2)}{(t-1)^2-x^2}.
$$
\end{corollary}

Another useful class of meanders that can be easily described is the class of iterated snakes.
\begin{definition}
Let $\mathcal{IS}_0$ be a set of snakes. For each $k=1,2,\dots$ let $\mathcal{IS}_k$ be a set of singular meanders that can be obtained by the insertion of a snake into a meander from $\mathcal{IS}_{r}$ for $r<k$. We say that $M$ is an \emph{iterated snake} if $M\in \mathcal{IS}_k$ for some $k$.
\end{definition}

\begin{thm}
Let $M_{n,k}^{(IS)}$ be the number of equivalence classes of iterated snakes of order $(n,k)$, and let $\phi_{(IS)}(x,t) = \sum\limits_{n\geq 0, k\geq 0}M_{n,k}^{(IS)}x^n t^k$ be the generating function for these numbers. 
Then the following equation holds:
\begin{equation}\label{eq:iter snakes}
%\phi_{(IS)}(x,t)=\phi_{(S)}\boxdot \left(\frac{\phi_{(IS)}^{(1)}(x,t) - x}{2}\right)(x,t)
 \phi_{(IS)}(x,t)=\phi_{(S)}\left(x + \frac{\phi_{(IS)}^{(1)}(x,t) - x}{2},t\right).  
\end{equation}
\end{thm}

\begin{proof}
Let $\phi_{r}(x,t)$ be the generating function for $\mathcal{IS}_r$ and $\phi_{0}(x,t):=\phi_{(S)}(x,t)$. Singular meanders from the set $\mathcal{IS}_r$ are obtained by the insertion of singular meanders from the set $\mathcal{IS}_{r-1}$ into a snake. 
If we forbid the insertion of direct snakes\footnote{We say that a singular meander $M$ of total order $N$ is a \emph{direct} (resp. \emph{inverse}) \emph{snake} if its permutation is $(1,2,\dots,N)$ (resp. $(N,N-1,\dots,1)$.} into other snakes then each singular meander from $\mathcal{IS}_r$ is constructed in a unique way. Iterated snakes from $\mathcal{IS}_r$ that are obtained by a sequence of the insertions of inverse snakes into another inverse snakes have the generating function 
$$
\frac{-x+\sum\limits_{n\geq 0, k\geq 0}M_{2n+1,k}^{(IS_r)}x^{2n+1}t^k}{2} = \frac{\phi_{r}^{(1)}(x,t) - x}{2}.
$$ Thus we have the following equality:
\begin{equation}\label{eq:iter snakes 2}
\phi_{r}(x,t)=\phi_{(S)}\left(x + \frac{\phi_{r-1}^{(1)}(x,t) - x}{2},t\right).
\end{equation}
The passage to the limit in~\eqref{eq:iter snakes 2} gives us the equation~\eqref{eq:iter snakes}.
\end{proof}

\begin{remark}
The solution of~\eqref{eq:iter snakes} can be found directly. However, the exact formula is too cumbersome and we do not put it here. 
\end{remark}
\begin{remark}
The numbers $M_{n,0}^{(IS)}$ are the same as the numbers of P-graphs with $2n$ edges defined in the work~\cite{R86}.
\end{remark}

\subsection{Irreducible singular meanders}
\begin{definition}
We say that a singular meander $M$ of total order $N$ is \emph{irreducible} if there are precisely $N+2$ submeanders of $M$ that are pairwise non-equivalent with respect to $M$.
\end{definition}
\begin{remark}
Examples of irreducible singular meanders are presented in Fig.~\ref{fig: example of irreducible}.
\end{remark}
We denote by $M^{(IR)}_{n,k}$ the number of equivalence classes of irreducible meanders of order $(n,k)$. We also denote by $\phi_{(IR)}(x,t) := \sum\limits_{n\geq 0, k \geq 0}M^{(IR)}_{n,k}x^n t^k$ the generating function for these numbers.

\begin{figure}[h]
\begin{tikzpicture}[scale = 1.3]
\draw[thick] (0, 0) to (3, 0);
\draw[ultra thick] (0.0501479, 0.384615) 
 to[out = 0, in = 90, distance = 8.69979] (0.692308, 0.1)
 to[out = -90, in = -90, distance = 2.89993] (0.923077, 0.1)
 to[out = 90, in = 90, distance = 8.69979] (1.61538, 0)
 to[out = -90, in = -90, distance = 8.69979] (2.30769, -0.1)
 to[out = 90, in = 90, distance = 2.89993] (2.53846, -0.1)
 to[out = -90, in = -90, distance = 14.4997] (1.38462, -0.1)
 to[out = 90, in = 90, distance = 2.89993] (1.15385, -0.1)
 to[out = -90, in = -90, distance = 8.69979] (0.461538, -0.1)
 to[out = 90, in = 90, distance = 2.89993] (0.230769, -0.1)
 to[out = -90, in = -90, distance = 31.8992] (2.76923, 0)
 to[out = 90, in = 90, distance = 8.69979] (2.07692, 0.1)
 to[out = -90, in = -90, distance = 2.89993] (1.84615, 0.1)
to[out = 90, in = 180, distance = 14.4997] (2.94985, 0.384615);
\draw[help lines] (1.5, 0) circle (1.5);
\draw[fill] (0.0501479, 0.384615) circle (0.05);
\draw[fill] (2.94985, 0.384615) circle (0.05);
\draw[fill] (0, 0) circle (0.05);
\draw[fill] (3, 0) circle (0.05);
\end{tikzpicture}
\hspace{1.2cm}
\begin{tikzpicture}[scale = 1.3]
\draw[thick] (0, 0) to (3, 0);
\draw[ultra thick] (0.17335, 0.7) 
to[out = 0, in = 90, distance = 25.3894] (2.1, 0.12)
 to[out = -90, in = -90, distance = 3.76991] (1.8, 0.12)
 to[out = 90, in = 90, distance = 18.8496] (0.3, 0.120)
 to[out = -90, in = -90, distance = 3.76991] (0.6, 0.120)
 to[out = 90, in = 90, distance = 11.3097] (1.5, 0)
 to[out = -90, in = -90, distance = 11.3097] (2.4, -0.120)
 to[out = 90, in = 90, distance = 3.76991] (2.7, -0.120)
 to[out = -90, in = -90, distance = 18.8496] (1.2, -0.120)
 to[out = 90, in = 90, distance = 3.76991] (0.9, -0.120)
to[out = -90, in = 180, distance = 25.3894] (2.82665, -0.7);
\draw[help lines] (1.5, 0) circle (1.5);
\draw[fill] (0.17335, 0.7) circle (0.05);
\draw[fill] (2.82665, -0.7) circle (0.05);
\draw[fill] (0, 0) circle (0.05);
\draw[fill] (3, 0) circle (0.05);
\end{tikzpicture}
\caption{Examples of ireducible meanders}
\label{fig: example of irreducible}
\end{figure}
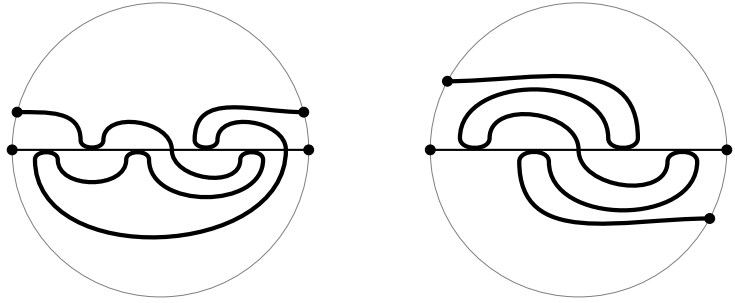

\begin{remark}
There are no irreducible meanders of order $(n,0)$ for $n\geq 0$ as there are always at least two consecutive numbers in the permutation of a non-singular meander.   
\end{remark}

The author does not know any suitable way to enumerate irreducible meanders. Several estimates on the growth rate were obtained in the work~\cite{BM21}. The numbers $M^{(IR)}_{n,k}$ for $n+2k\leq 36 $ where calculated in the work~\cite{B21}. 

\section{Decomposition of singular meanders}
\begin{thm}\label{thm:decompose}
Each singular meander is canonically constructed from irreducible meanders and iterated snakes using the operations~\eqref{operation 1} and \eqref{operation 2}.
\end{thm}
\begin{proof}
We say that a submeander $M'$ of a singular meander $M$ is \emph{a prime factor} if either $M'$ is an irreducible singular meander, or $M'$ is an iterated snake and there are no submeander $M''$ of $M$ such that $M''$ is an iterated snake and $M' \prec M'' \prec M$. 

There is a canonical way to decompose an arbitrary singular meander into prime factors. We start with a singular meander $M_0$, we cut all its prime factors and obtain a singular meander $M_1$, then we cut all prime factors of $M_1$ and so on, until we get an irreducible singular meander or an iterated snake.

If we are going to construct any singular meanders from irreducible singular meanders and iterated snakes in a unique way we only need to permit insertions of an iterated snake into another iterated snake. 
\end{proof}

We can use this theorem to obtain an equation on the generating function for singular meanders. 

\begin{corollary}\label{cor:gen function equation}
Let $\psi(x,t) = \sum\limits_{n\geq 0, k \geq 0} M_{n,k}x^n t^k$ be the generating function for the numbers of equivalence classes of singular meanders. Then the following equation holds:
\begin{align*}
    \psi(x,t) = \left(\phi_{(IR)}\boxdot \psi\right)(x,t) + \left(\phi_{(IS)}\boxdot\left(x+t+\left(\phi_{(IR)}\boxdot \psi\right)\right)\right)(x,t).
\end{align*}
\end{corollary}

%Moreover there exists the canonical way of constructing an arbitrary singular meanders from iterated snakes and irreducible meanders.  
\bibliography{biblio}
\bibliographystyle{halpha.bst}
\end{document}